\documentclass[12pt]{amsart}

\usepackage{graphicx}

\usepackage{enumerate}

\newtheorem{thm}{Theorem}[section]
\newtheorem*{thm*}{Theorem}

\newtheorem{conj*}{Conjecture}
\newtheorem{lemma}[thm]{Lemma}
\newtheorem*{lemma*}{Lemma}
\newtheorem{prop}[thm]{Proposition}
\newtheorem*{prop*}{Proposition}
\newtheorem{cor}[thm]{Corollary}

\theoremstyle{definition}
\newtheorem{df}[thm]{Definition}

\newtheorem*{ex}{Example}

\theoremstyle{remark}
\newtheorem*{rmk}{Remark}

\newtheorem*{claim*}{Claim}

\newcommand{\bbQ}{\mathbb{Q}}

\newcommand{\bbF}{\mathbb{F}}
\newcommand{\bbC}{\mathbb{C}}
\newcommand{\bbZ}{\mathbb{Z}}
\newcommand{\bbP}{\mathbb{P}}

\newcommand{\bbR}{\mathbb{R}}

\newcommand{\QQ}{\bbQ}

\newcommand{\RR}{\bbR}
\newcommand{\CC}{\bbC}
\newcommand{\FF}{\bbF}

\newcommand{\ZZ}{\bbZ}

\newcommand{\PP}{\bbP}

\newcommand{\cO}{\mathcal{O}}

\newcommand{\cH}{\mathcal{H}}

\newcommand{\fp}{\mathfrak p}

\DeclareMathOperator{\rk}{rk}

\DeclareMathOperator{\End}{End}

\DeclareMathOperator{\SL}{SL}
\DeclareMathOperator{\PSL}{PSL}

\def\det{\operatorname{det}}

\DeclareMathOperator*{\mysum}{\sum}

\newcommand{\emphh}[2][ ]{%
\ifthenelse{\equal{#1}{ }}{\index{default}{#2} {\emph{#2}}}{\index{default}{#1@#2} {\emph{#2}}}%
}

\def\sumprime{\mathop{\sum{\raise3pt\hbox{${}'$}}}} 

\newcommand{\tto}[1]{%
\ifthenelse{\equal{#1}{}}{\to}{\stackrel{#1}{\to}}}

\newcommand{\fixme}[1]{}

\usepackage[all]{xy}
\xyoption{arc}

\usepackage{mathrsfs}
\usepackage{wrapfig, longtable}

\DeclareMathOperator{\Fr}{Fr}

\begin{document}

\title[The non-ordinary locus of the TTV curves]{The non-ordinary locus of the TTV family of curves}

\author{Luiz Kazuo Takei}

\address{Department of Mathematics, John Abbott College, Sainte-Anne-de-Bellevue, QC, Canada}

\email{luiz.kazuotakei@johnabbott.qc.ca}


\maketitle

\vspace{\baselineskip}

\begin{abstract}
This article establishes a relation between the non-ordinary locus of the reduction mod $p$ of the TTV family of curves and the genus of certain triangular modular curves.
\end{abstract}

\section{Introduction}

Let $\Gamma_0(p) = \left\{ M \in \SL_2(\ZZ) \ \left| \ M \equiv \left( \begin{smallmatrix} * & * \\ 0 & * \end{smallmatrix} \right) \pmod{\fp} \right. \right\} \subseteq \SL_2(\ZZ)$ and $\cH^* = \{ z \in \CC \ | \ \Im(z) > 0 \} \cup \QQ$. It is well-known that the modular curve $X_0(p) = \Gamma_0(p) \backslash \cH^*$ admits an integral model for which the reduction modulo $p$ consists of two copies of $X_0(1)_{\FF_p} = \PP^1_{\FF_p}$ crossing transversally at the points that represent supersingular elliptic curves.

This article's motivating question is whether one can find a similar relation when the classical modular group $\SL_2(\ZZ)$ is replaced by the triangle group $\Gamma_{5, \infty, \infty}$ and elliptic curves are replaced by TTV curves (cf. Definition \ref{def:ttv_curves} below). The reason why TTV curves are chosen here is due to a connection between those curves and $\Gamma_{5, \infty, \infty}$ found by Henri Darmon via Frey representations (cf. Definition 1.1 and Theorem 1.10 in \cite{dar2000} and Definition 8 in \cite{darmon2004}) .

The definitions of $\Gamma_{5, \infty, \infty}$ and the TTV curves are given in sections 1 and 2. The main results lie in Section 4, where Theorem \ref{thm:genus_and_non_ordinary_elmts} states a relation between the genus of a version of $X_0(p)$ for $\Gamma_{5, \infty, \infty}$ and the number of non-ordinary TTV curves mod $p$, providing evidence that the answer to the motivating question might be positive.

\subsection*{Acknowledgements:} The motivating question was raised by Eyal Goren, with whom I had helpful discussions. I would also like to thank an anonymous referee that provided useful comments that helped improve and clarify this article.

\section{Triangle Groups}

In this section, a particular triangle group and some important subgroups are defined. For more details and a more general description, cf. section 2 of \cite{takei2012}.

The triangle group $\Gamma_{5, \infty, \infty}$ is defined to be the subgroup of $\SL_2(\RR)$ generated by
\[
	\gamma_1 = \begin{pmatrix}
					-2 \cos(\frac{\pi}{5}) & -1 \\
					1 & 0
				\end{pmatrix}
	\ \textrm{ and } \
	\gamma_2 = \begin{pmatrix}
					0 & 1 \\
					-1 & 2
				\end{pmatrix}.
\]

In what follows, whenever $\Gamma \subseteq \SL_2(\RR)$, the group $\overline{\Gamma} \subseteq \PSL_2(\RR)$ will denote the image of $\Gamma$ in $\PSL_2(\RR)$.

It is a fact that $\Gamma_{5, \infty, \infty}$ is a Fuchsian group and, moreover, $\left. \overline{\Gamma_{5, \infty, \infty}} \right\backslash \cH^{*} \cong \PP^1$ where $\cH^* = \cH \cup \{ \textrm{cusps of } \Gamma_{5, \infty, \infty} \}$.

Note that $\Gamma_{5, \infty, \infty}$ is a subgroup of $\SL_2(\cO)$, where $\cO = \ZZ[\zeta_{10} + \zeta_{10}^{-1}]$ and $\zeta_{10} = e^{2 \pi i / 10}$.

\begin{df}
    Given a prime ideal $\fp$ of $\cO$, the \textit{congruence subgroups} of $\Gamma_{5, \infty, \infty}$ with level $\fp$ are defined to be
    \[
        \Gamma_{5, \infty, \infty}(\fp) = \left\{ M \in \Gamma_{5, \infty, \infty} \ \left| \ M \equiv \begin{pmatrix} 1 & 0 \\ 0 & 1 \end{pmatrix} \pmod{\fp} \right. \right\}, \textrm{ and }
    \]
    \[
        \Gamma_{5, \infty, \infty}^{(0)}(\fp) = \left\{ M \in \Gamma_{5, \infty, \infty} \ \left| \ M \equiv \begin{pmatrix} * & * \\ 0 & * \end{pmatrix} \pmod{\fp} \right. \right\}.
    \]
\end{df}

\section{The TTV curves}

\begin{df}
	\label{def:ttv_curves}
	The TTV curves are defined to be the following:
	\begin{equation}
	\label{eqn:TTV_family_of_curves}
	    \begin{array}{rcl}
    	    C^-(t) & : & y^2 = x^5 - 5x^3 + 5x + 2 - 4t \\
	        \\
        	C^+(t) & : & y^2 = (x+2) ( x^5 - 5x^3 + 5x + 2 - 4t ).
    	\end{array}
	\end{equation}
\end{df}

In \cite{ttv1991}, W. Tautz, J. Top, and A. Verberkmoes studied the two families of hyperelliptic
curves defined above. In particular, it was shown that their Jacobians have real multiplication by $\cO_L$, the ring of integers of $L = \QQ(\zeta_{10})^+ = \QQ(\zeta_{10} + \zeta_{10}^{-1})$.

\section{The Hasse-Witt and Cartier-Manin matrices}

This section is mainly based on Chapters 9 and 10 of \cite{serre1958} and \cite{yui1978}.

\subsection{Hasse-Witt matrix}

Let $k$ be a perfect field of characteristic $p > 2$ and $C$ a hyperelliptic curve of genus $g > 0$ defined over $k$. This notion can be defined in a more general context but we will focus on hyperelliptic curves.

\begin{df}
    Fix a basis of $H^1(C, \cO_C)$. The \emph{Hasse-Witt matrix} of $C$ is the matrix of the $p$-linear operator $F : H^1(C, \cO_C) \rightarrow H^1(C, \cO_C)$, where $F$ is the Frobenius operator.
\end{df}

\begin{rmk}
    Notice that the Hasse-Witt matrix is dependent on the basis chosen. Because of the $p$-linearity of the Frobenius operator, if $H$ and $H'$ are Hasse-Witt matrices with respect to different bases, then there is a matrix $U$ such that
    \[
        H' = U^{-1} H U^{(p)},
    \]
    where $U^{(p)}$ is the matrix obtained from $U$ by raising all its entries to the $p$-th power.
\end{rmk}

There is another way to essentially define the Hasse-Witt matrix of a curve. This is done in terms of the so called Cartier operator, which is studied in the next section.

\subsection{Cartier-Manin Matrix}

Suppose $C$ is given by
\begin{equation}
\label{eqn:curve_equation}
    y^2 = f(x)
\end{equation}
where $f(x)$ is a polynomial over $k$ without multiple roots of degree $2g + 1$.

Every element of $\Omega^1_C$ can be written as
\[
    \omega = d \varphi + \eta^p x^{p-1} dx
\]
for some $\varphi, \eta \in k(C)$.

\begin{df}
    The \emph{Cartier operator} $\mathscr{C} : H^0(C, \Omega^1_C) \rightarrow H^0(C, \Omega^1_C)$ is defined by
    \[
        \mathscr{C}(d \varphi + \eta^p x^{p-1} dx) = \eta dx.
    \]
\end{df}

\begin{df}
    The \emph{Cartier-Manin matrix} is the matrix of the $1/p$-linear operator $\mathscr{C} : H^0(C, \Omega^1_C) \rightarrow H^0(C, \Omega^1_C)$.
\end{df}

\begin{rmk}
    Because of the $1/p$-linearity of the operator $\mathscr{C}$, if $M$ and $M'$ are Cartier-Manin matrices with respect to different bases, then there is a matrix $U$ such that
    \[
        M' = U^{-1} M U^{(1/p)},
    \]
    where $U^{(p)}$ is the matrix obtained from $U$ by raising all its entries to the $p$-th power.
\end{rmk}

\begin{rmk}
    The Cartier operator, as defined here, is called the \emph{modified Cartier operator} in \cite{yui1978}. Moreover, the definition of the Cartier-Manin matrix given by N. Yui is slightly different (cf. page 381 of of \cite{yui1978}).
\end{rmk}

The relation between the Hasse-Witt matrix and the Cartier-Manin matrix arises as follows. It is known that $H^0(C, \Omega^1_C)$ is the dual of $H^1(C, \cO_C)$. Under this identification, the following result (cf. Prop. 9, Section 10 in \cite{serre1958}) holds.

\begin{prop}
    The map $\mathscr{C} : H^0(C, \Omega^1_C) \rightarrow H^0(C, \Omega^1_C)$ is the dual of $F : H^1(C, \cO_C) \rightarrow H^1(C, \cO_C)$.
\end{prop}

N. Yui (cf. pages 380-381 in \cite{yui1978}) gives a concrete way of computing the Cartier-Manin matrix of a curve:

\begin{prop}
\label{prop:explicit_cm_matrix}
    Let $C$ be given by (\ref{eqn:curve_equation}). Then the Cartier-Manin matrix of $C$ with respect to the basis
    \[
        \frac{dx}{y} \ , x \frac{dx}{y} \ , \ \dotsc \ , \ x^{g-1} \frac{dx}{y}
    \]
    of $H^0(C, \Omega^1_C)$ is given by
    \[
        N^{(1/p)},
    \]
    where
    \[
        N = (c_{ip - j}) =
            \begin{pmatrix}
                c_{p-1} & c_{p - 2} & \dotsc & c_{p-g} \\
                c_{2p-1} & c_{2p - 2} & \dotsc & c_{2p-g} \\
                \dotsc \\
                c_{gp-1} & c_{gp - 2} & \dotsc & c_{gp-g} \\
            \end{pmatrix},
    \]
    and
    \[
        f(x)^{(p-1)/2} = \mysum c_r x^r.
    \]
\end{prop}

\subsection{Jacobian of $C$}

From now on, $k$ will be a finite field of characteristic $p>2$.

Recall the following definitions:

\begin{df}
    An abelian variety $A$ of dimension $g$ over $k$ is called
    \begin{itemize}
        \item \emph{ordinary} if its $p$-rank is $g$, i.e., $\#(A[p]) = p^g$;
        \item \emph{supersingular} if $A$ is $\overline{k}$-isogenous to a power of a supersingular elliptic curve.
    \end{itemize}
\end{df}

\begin{rmk}
    As explained in Section 3.2 of \cite{zhu2000}, if $A$ is supersingular, then $A[p] = 0$. The converse holds if $g = 1$ or $2$ but not necessarily if $g > 2$.
\end{rmk}

Let $J = J(C)$ be the Jacobian of the curve $C$.

The results below show the relation between the Cartier-Manin matrix of $C$ and $J$.

\begin{prop}
    The $p$-rank of $J$ is bounded above by the rank of the Cartier-Manin matrix, i.e., $\sigma \leq \rk(M)$, where $\#(J[p]) = p^{\sigma}$ and $M$ denotes the Cartier-Manin matrix.
\end{prop}
\begin{proof}
    This is a corollary of Proposition 10 in Section 11 of \cite{serre1958}.
\end{proof}

\begin{prop}
\label{prop:type_of_curve}
    Let $M$ be the Cartier-Manin matrix of $C$ and $N = M^{(p)}$. The following holds:
    \begin{enumerate}[(a)]
        \item $\det(N) \neq 0$ if and only if $J$ is ordinary.
        \item $N = 0$ if and only if $J$ is a product of supersingular elliptic curves.
        \item If the genus of $C$ is $2$, then $N^{(p)} N = 0$ if and only if $J$ is supersingular.
    \end{enumerate}
\end{prop}
\begin{proof}
    Cf. \cite{yui1978} (Theorems 3.1 and 4.1), \cite{nygaard1981} (Theorem 4.1) and \cite{manin1963} (p. 78).
\end{proof}

\subsection{Curves with real multiplication}

Let $L$ be a totally real number field such that $[L:\QQ] = g$ and $p$ a prime number that is unramified in $L$. In this subsection $C$ will denote a projective algebraic curve of genus $g$ and $J$ its Jacobian, which is assumed to have \emph{real multiplication by $\cO_L$}, that is, with an embedding of rings
\[
    \iota: \cO_L \longrightarrow \End(J)
\]
as explained in definition 2.2.1 of \cite{goren2002}.

In this section, the Cartier operator $\mathscr{C}$ (hence, the Cartier-Manin matrix) is studied via the corresponding operator on the Jacobian of $C$.

\begin{center}
\begin{tabular}{|c|c|}
    \hline
    $C$ & $J$ \\
    \hline
    action of $\mathscr{C}$ on $H^0(C, \Omega^1_C)$ & action of $V$ on $H^0(J, \Omega^1_J)$  \\
    action of $F$ on $H^1(C, \cO_C)$ & action of $F$ on $H^1(J, \cO_J)$ \\
    \hline
\end{tabular}
\end{center}

The vector spaces on the left column are isomorphic to the ones on the right column. Furthermore, the semi-linear operators on the left column coincide (via that isomorphism) to the ones on the right column.

\begin{thm}
\label{thm:structure_cm_matrix}
    As an $\cO_L \otimes k$-module, the space $H^1(J, \cO_J)$ decomposes as
    \[
        H^1(J, \cO_J) = \bigoplus\limits_{\sigma \in B} W_{\sigma},
    \]
    where
    \[
        B = \{ \sigma: \cO_L \rightarrow k \mid \sigma \textrm{ a ring homomorphism} \} \ \textrm{ and } \ \dim_k W_{\sigma} = 1.
    \]
    Moreover, the action of $F$ commutes with the action of $\cO_L \otimes k$ and satisfies the following
    \[
        F(W_{\sigma}) \subseteq W_{\Fr \circ \sigma}.
    \]
\end{thm}
\begin{proof}
    Cf. Lemma 2.3.1 and Remark 2.2.8 in \cite{goren&oort2000}.
\end{proof}

\begin{rmk}
    Being the dual of $F$, a similar statement holds for the action of $V$ on $H^0(J, \Omega^1_J)$.
\end{rmk}

\begin{rmk}
    Consider the factorization of $p$ in $\cO_L$ given by
    \[
        p \cO_L = \fp_1 \fp_2 \dotsb \fp_r.
    \]

    Then, it is not hard to see that, with the notation of Theorem \ref{thm:structure_cm_matrix}, $B$ decomposes as
    \[
        B = B_1 \sqcup B_2 \sqcup \dotsb \sqcup B_r,
    \]
    where $\#B_i = f = f(\fp_i / p) = [ \cO_l / \fp_i : \FF_p ]$. Furthermore, $\Fr$ acts transitively on each $B_i$, i.e.,
    \[
        B_i = \{ \sigma_i = \Fr^f \circ \sigma_i \ , \ \Fr \circ \sigma_i \ , \ \dotsc \ , \ \Fr^{f - 1} \circ \sigma_i \}.
    \]
\end{rmk}

\section{Studying $C^{\pm}(t)$}
\label{sec:non_ordinary_locus_TTV}

In this section we return to the families of curves defined in (\ref{eqn:TTV_family_of_curves}). These curves have genus $g=2$ and, as was mentioned in the previous chapter, they have real multiplication by $\cO_L$ (where $L = \QQ(\sqrt{5})$).

Consider a prime $p > 2$ that is unramified in $L$.

\begin{lemma} There are only two possibilities for such a $p$:
\label{lem:characterization_of_p}
    \begin{itemize}
        \item $p$ is a product of two primes in $\cO_L$ (when $p \equiv 1, 4 \pmod{5}$); or
        \item $p$ is inert in $\cO_L$ (when $p \equiv 2,3 \pmod{5}$).
    \end{itemize}
\end{lemma}
\begin{proof}
     Cf. (1.1) in Chapter V of \cite{frohlich&taylor1993}.
\end{proof}

\subsection{The Cartier-Manin matrix of the curve $C^-$}

Example 3.5 (or the proof of the main result) in \cite{ttv1991} shows that the action of $\cO_L$ on $H^0(C^-, \Omega^1)$ has two distinct eigenvectors, namely:
\[
    \frac{dx}{y}, x \frac{dx}{y}
\]

Thus, Lemma \ref{lem:characterization_of_p}, Theorem \ref{thm:structure_cm_matrix} and the remarks that follow it yield the result below.

\begin{thm}
\label{thm:cm_matrix_c-}
    The Cartier-Manin matrix of $C^-$ with respect to the basis $\{ \frac{dx}{y}, x \frac{dx}{y} \}$ of $H^0(C^-, \Omega^1)$ is given by
    \[
        M = \left\{
                \begin{array}{ll}
                    \begin{pmatrix}
                        * & 0 \\
                        0 & *
                    \end{pmatrix},
                    &
                    \textrm{ if } p \equiv 1, 4 \pmod{5}
                    \\
                    \begin{pmatrix}
                        0 & * \\
                        * & 0
                    \end{pmatrix},
                    &
                    \textrm{ if } p \equiv 2, 3 \pmod{5}.
                \end{array}
            \right.
            ,
    \]
    where $*$ are elements of $\FF_p[t]$.
\end{thm}

\begin{rmk}
    A curious consequence of this fact is the following non-trivial result. Let $p$ be a prime number such that $p \neq 2, 5$, $f(x) = x^5 - 5x^3 + 5x + 2 - 4t \in \ZZ[t][x]$ and $f(x)^{(p-1)/2} = \mysum c_r x^r$. If
    \begin{itemize}
        \item $p \equiv 1, 4 \pmod{5}$, then
            \[
                c_{p-1} \equiv c_{2p-2} \equiv 0 \pmod{p}
            \]
        \item $p \equiv 2, 3 \pmod{5}$, then
            \[
                c_{p-2} \equiv c_{2p-1} \equiv 0 \pmod{p}.
            \]
    \end{itemize}
\end{rmk}
\begin{proof}
    Follows from the previous result and Proposition \ref{prop:explicit_cm_matrix}.
\end{proof}

\begin{cor}
\label{cor:c-_structure}
    If $p \equiv 2$ or $3 \pmod{5}$, then the Jacobian of the curve $C^-$ is either supersingular or ordinary.
\end{cor}
\begin{proof}
    This is a direct consequence of the previous theorem and of Proposition \ref{prop:type_of_curve}.

    In fact, using Proposition \ref{prop:type_of_curve}, we have that the Jacobian $J^-$ of $C^-$ is ordinary if and only if $\det(N) \neq 0$, where $N = M^{(p)}$ and $M$ is the Cartier-Manin matrix of $C^-$. Also, since $r=5$ (and, thus, the genus of $C^-$ is $2$), $J^-$ is supersingular if and only if $N^{(p)} N = 0$.
    So it suffices to check that $\det(N) = 0$ if and only if $N^{(p)} N = 0$. This follows easily from the previous theorem.
\end{proof}

\subsection{The Cartier-Manin matrix of the curve $C^+$}

Following the ideas of the proof of the main result of \cite{ttv1991}, one can compute the action of $\cO_L$ on $H^0(C^+, \Omega^1)$.

\begin{prop}
    The Jacobian of the curve $C^+$ has real multiplication by $\cO_L$. Moreover, the action of $\cO_L$ on $H^0(C^+, \Omega^1)$ has a basis of eigenvectors, namely:
    \[
        dx/y \ , \ dx/y + y dx/y
    \]
\end{prop}
\begin{proof}

    Tautz-Top-Verberkmoes (\cite{ttv1991}) showed that $C^+$ is the quotient $D_t / \sigma$, where
    \[
        D_t : y^2 = x^{10} + tx^5 + 1
    \]
    and $\sigma \in \End(D_t)$ defined by
    \[
        \sigma: (x,y) \mapsto (1/x, y/x^5).
    \]

    Using that
    \[
        X^{2n} + 1 = X^n (X+X^{-1}) \cdot g(X^2 + X^{-2}) \in k[X,X^{-1}]
    \]
    for any odd $n$, it follows that the map $\varphi: D_t \rightarrow C^+$ given by
    \[
        \varphi: (x,y) \mapsto (x + 1/x, y(x+1)/x^3)
    \]
    is well-defined and corresponds to the natural quotient map $D_t \rightarrow D_t / \sigma$. Moreover, it makes the diagram below commutative
    \[
        \xymatrix{
            D_t \ar[r]^{\varphi} \ar[d] & C^+ \ar[d] \\
            \PP^1 \ar[r] & \PP^1
        },
    \]
    where
    \[
        \begin{array}{ccc}
                \PP^1 & \rightarrow & \PP^1 \\
                x & \mapsto & x + 1/x
        \end{array}
    \]
    and the vertical maps are just
    \[
        (x,y) \mapsto x.
    \]

    The curve $D_t$ has multiplication by $\cO_{\QQ(\zeta_5)}$ coming from the map
    \[
        \zeta: (x,y) \mapsto (\zeta_5x, y).
    \]

    To prove that the Jacobian of $C^+$ has multiplication by $\cO_L$, it is enough to show that the action of $\zeta^* + (\zeta^{-1})^*$ preserves the space $(\Omega^1_{D_t})^{\sigma}$ of $\sigma$-invariant differentials of $D_t$. One checks that a basis for $(\Omega^1_{D_t})^{\sigma}$ is given by
    \[
        \omega_1 = (x^2-x) dx/y \ , \ \omega_2 = (x^3-1) dx/y.
    \]

    Now, by the definition of $\zeta$, one computes that
    \[
        [ \zeta^* + (\zeta^{-1})^* ] \omega_1 = (\zeta_5^2 + \zeta_5^{-2}) \omega_1
    \]
    and
    \[
        [ \zeta^* + (\zeta^{-1})^* ] \omega_2 = (\zeta_5 + \zeta_5^{-1}) \omega_2.
    \]

    Now it remains only to show that
    \[
        dx/y \ , \ dx/y + y dx/y \ \in \Omega^1_{C^+}
    \]
    are eigenvectors for the action of $\cO_L$.

    Notice that the action on $\Omega^1_{C^+}$ comes from the action on $(\Omega^1_{D_t})^{\sigma}$ (these spaces are identified via $\varphi$). Now, the definition of $\varphi$ yields
    \[
        \begin{array}{c}
            \varphi^* (dx/y) = \omega_1 \\
            \varphi^* (dx/y + y dx/y) = \omega_2.
        \end{array}
    \]

    From the previous computations, this finishes the proof.
\end{proof}

This result and Lemma \ref{lem:characterization_of_p} yield

\begin{thm}
    The Cartier-Manin matrix of $C^+$ with respect to the basis $\{ \frac{dx}{y}, x \frac{dx}{y} \}$ of $H^0(C^-, \Omega^1)$ is given by
    \[
        M = \left\{
                \begin{array}{ll}
                    \begin{pmatrix}
                        a & b - a \\
                        0 & b
                    \end{pmatrix},
                    &
                    \textrm{ if } p \equiv 1, 4 \pmod{5}
                    \\
                    \begin{pmatrix}
                        a & b \\
                        a & -a
                    \end{pmatrix},
                    &
                    \textrm{ if } p \equiv 2, 3 \pmod{5}.
                \end{array}
            \right.
    \]
    for some
    \[
        a, b \in \FF_p[t].
    \]
\end{thm}

\begin{cor}
    If $p \equiv 2$ or $3 \pmod{5}$, then the Jacobian of the curve $C^+$ is either supersingular or ordinary.
\end{cor}
\begin{proof}
    This is a direct consequence of the previous theorem and of Proposition \ref{prop:type_of_curve}. The proof is similar to the proof of corollary \ref{cor:c-_structure}.
\end{proof}

\subsection{A relation between $X_{5, \infty, \infty}^{(0)}(\fp)$ and the family $C^-$}
\label{ssc:genus_and_non_ordinary_elmts}

\begin{wrapfigure}{l}{0.5\textwidth}
    \includegraphics[scale=0.5]{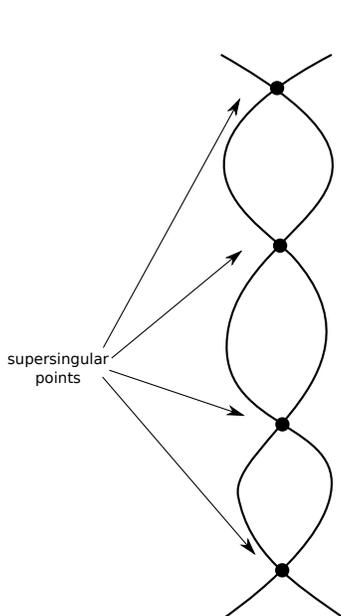}
    \caption{Reduction of $X_0(p)$ modulo $p$}
    \label{fig:X_0(p)_modp}
\end{wrapfigure}

It is known that $X_0(p) = \Gamma_0(p) \backslash \cH^*$ admits an integral model for which the reduction modulo $p$ consists of two copies of $X_0(1)_{\FF_p} = \PP^1_{\FF_p}$ crossing transversally at the supersingular points as shown in figure \ref{fig:X_0(p)_modp} (cf. Theorem 6.9, page DeRa-144, in \cite{deligne&rapoport1973}). In particular, there is a relation between the genus of $X_0(p)$ and the number of supersingular elliptic curves modulo $p$.

In this subsection we investigate a similar property for the triangular modular curve $X_{5, \infty, \infty}^{(0)}(\fp)$: we show that in certain cases, the genus of the curve $X_{5, \infty, \infty}^{(0)}(\fp)$ is closely related to the number of non-ordinary elements of the family of curves $C^-$. More specifically, the following result holds:

\begin{thm}
\label{thm:genus_and_non_ordinary_elmts}
    Let $p > 5$ be a prime number such that $p$ splits in $\cO_{\QQ(\sqrt{5})}$ (i.e., $p \equiv 1$ or $4 \pmod{5}$) and take $\fp$ a prime ideal above $p$. Furthermore, let $g$ be the genus of $X_{5, \infty, \infty}^{(0)}(\fp)$ and $d(t) = \det(M^{(p)})$, where $M$ is the Cartier-Manin matrix of $C^-$ (as computed in Theorem \ref{thm:cm_matrix_c-}). Then
    \[
        g = \deg(d(t)) + \delta,
    \]
    where
    \[
        \delta = \left\{
                    \begin{array}{ll}
                        -1, & \textrm{if } p \equiv 1 \pmod{5} \\
                        1, & \textrm{if } p \equiv 4 \pmod{5}.
                    \end{array}
                \right.
    \]
\end{thm}
\begin{proof}
    Since $p$ is assumed to be split, Proposition 4.5 in \cite{takei2015} implies that the genus of $X_{5, \infty, \infty}^{(0)}(\fp)$ is given by
    \[
        g = 2n - 1,
    \]
    where
    \[
        p + 1 = 5n + m
    \]
    with
    \[
        m = \left\{
            \begin{array}{ll}
                0, & \textrm{if } p \equiv -1 \pmod{5} \\
                2, & \textrm{if } p \equiv 1 \pmod{5}.
            \end{array}
        \right.
    \]

    Thus,
    \[
        g = \frac{2}{5} (p+1-m) - 1.
    \]

    It follows from Theorem \ref{thm:cm_matrix_c-} that the Cartier-Manin matrix is given by $\left( \begin{smallmatrix} * & 0 \\ 0 & * \end{smallmatrix} \right)$ , it suffices to compute the degree of the entries of the main diagonal, which is done in the next lemma.
\end{proof}

\begin{lemma}
    Let $p$ be as in the statement of the previous proposition and
    \[
        \begin{pmatrix}
            a(t) & 0 \\
            0 & b(t)
        \end{pmatrix}
    \]
    be the Cartier-Manin matrix of $C^-$ with respect to $p$. Then
    \[
        \deg(a(t)) = \left\{
                        \begin{array}{ll}
                            \frac{3}{2}k, & \textrm{if } p = 5k + 1 \\
                            \frac{3}{2}k - 1, & \textrm{if } p = 5k - 1 \\
                        \end{array}
                    \right.
    \]
    and
    \[
        \deg(b(t)) = \left\{
                        \begin{array}{ll}
                            \frac{1}{2}k, & \textrm{if } p = 5k + 1 \\
                            \frac{1}{2}k - 1, & \textrm{if } p = 5k - 1 \\
                        \end{array}
                    \right.
    \]
\end{lemma}
\begin{proof}
    By Proposition \ref{prop:explicit_cm_matrix}, $a(t)$ is the $(p-1)$-th coefficient of $f(x)^{(p-1)/2}$, where
    \[
        f(x) = x^5 - 5x^3 + 5x + 2 - 4t.
    \]
    Since this lemma is only concerned about the degree (with respect to $t$) of a certain coefficient, $f$ can be assumed to be
    \[
        f(x) = x^5 - 5x^3 + 5x + t.
    \]

    By the Multinomial Theorem,
    \[
        f(x)^{(p-1)/2} = \mysum_{a,b,c,d} (a,b,c,d)! \ (-1)^b \ 5^{b+c} \ x^{5a + 3b + c} \ t^d,
    \]
    where the sum is taken over all integers $a,b,c,d \geq 0$ such that $a+b+c+d = (p-1)/2$ and
    \[
        (a,b,c,d)! = \frac{((p-1)/2)!}{a! \ b! \ c! \ d!}.
    \]

    Therefore the $(p-1)$-th coefficient is given by
    \[
        a(t) = \mysum_{5a + 3b + c = p - 1} (a,b,c,d)! \ (-1)^b \ 5^{b+c} \ t^d.
    \]

    This implies that $\deg(a(t))$, at least over $\ZZ$, is given (possibly) by the largest $d$ such that
    \[
        d = 4a + 2b - \frac{(p-1)}{2}
    \]
    and
    \[
        \left\{
            \begin{array}{l}
                5a + 3b \leq p-1 \\
                a \geq 0 \ , \ b \geq 0.
            \end{array}
        \right.
    \]

    Assume now that that $p = 5k + 1$. One checks (using the graphical method of linear programming) that the solution is
    \[
        d = \frac{3}{2} k
    \]
    attained only once when
    \[
        \begin{array}{ccc}
            a = k & \textrm{ and } & b = 0.
        \end{array}
    \]

    Since this is attained only once, $\deg(a(t))$ over $\ZZ$ is actually $\frac{3}{2} k$. Using the fact that $p > 5$, it follows that the coefficient of the degree $\frac{3}{2} k$ term is not zero modulo $p$. Hence, $\deg(a(t)) = \frac{3}{2} k$ over $\FF_p$.

    A similar argument proves all the other cases. The only exception is the last case ($\deg(b(t))$ when $p = 5k - 1$), where the maximum $d$ is attained twice. But in this case a straight forward computation shows that the coefficient is still non-zero modulo $p$.
\end{proof}

\begin{rmk}
    Theorem \ref{thm:genus_and_non_ordinary_elmts} presents an interesting relation between the genus of $X_{5, \infty, \infty}^{(0)}(\fp)$ and the number of non-ordinary elements in the family $C^-$ modulo $p$ when $p$ is split. Unfortunately when $p$ is inert, the same does not hold. The example below shows that the difference between the degree of $d(t)$ and the genus of $X_{5,\infty,\infty}(\fp)$ grows with $p$ when $p$ is inert.

    It would be interesting to understand why there is this discrepancy between primes that are split and primes that are inert.
\end{rmk}

\begin{ex}
    Contrary to the split case, the difference between the genus of $X_{5, \infty, \infty}^{(0)}(\fp)$ and the degree of $d(t)$ is not $\pm 1$ when $p$ is inert. Here are the first few inert primes and their corresponding data as calculated using the computer algebra system SAGE (\cite{sage}):

    \begin{center}
        \begin{tabular}{| c | c | c | c |}
            \hline
            $p$ & genus of $X_{5, \infty, \infty}^{(0)}(\fp)$ & degree of $d(t)$ & genus - degree\\
            \hline
            7 & 13 & 2 & 11 \\
            13 & 55 & 4 & 51 \\
            17 & 99 & 6 & 93 \\
            23 & 189 & 8 & 181 \\
            37 & 511 & 14 & 497 \\
            43 & 697 & 16 & 681 \\
            47 & 837 & 18 & 819 \\
            53 & 1071 & 20 & 1051 \\
            67 & 1729 & 26 & 1703 \\
            73 & 2059 & 28 & 2031 \\
            83 & 2673 & 32 & 2641 \\
            97 & 3667 & 38 & 3629 \\
            103 & 4141 & 40 & 4101 \\
            \hline
        \end{tabular}
    \end{center}
\end{ex}

\vspace{\baselineskip}

\begin{rmk}
    Note that Theorem \ref{thm:genus_and_non_ordinary_elmts} actually describes a relation between the genus of $X_{5, \infty, \infty}^{(0)}(\fp)$ and the degree of $d(t)$, which is not exactly the same as the number of non-ordinary elements in the $C^-$ family. In our computations, summarized in the table below, the difference between the degree and the exact number of non-ordinary elements seems to be reasonably small. It would be interesting to understand how the difference grows and check whether it is bounded or not.
	\begin{center}
        \begin{longtable}{| c | c | c | c |}
            \hline
            $p$ & degree of $d(t)$ & \# of non-ordinary curves & difference \\
            \hline
			11 & 4 & 3 & 1\\
            19 & 6 & 5 & 1\\
            29 & 10 & 10 & 0\\
            31 & 12 & 11 & 1\\
            41 & 16 & 16 & 0\\
            59 & 22 & 21 & 1\\
            61 & 24 & 22 & 2\\
            71 & 28 & 25 & 3\\
            79 & 30 & 27 & 3\\
            89 & 34 & 32 & 2\\
            101 & 40 & 38 & 2\\
            109 & 42 & 42 & 0\\
            131 & 52 & 45 & 7\\
            139 & 54 & 53 & 1\\
            149 & 58 & 54 & 4\\
            151 & 60 & 57 & 3\\
            179 & 70 & 69 & 1\\
            181 & 72 & 68 & 4\\
            191 & 76 & 75 & 1\\
            199 & 78 & 75 & 3\\
            211 & 84 & 79 & 5\\
            229 & 90 & 90 & 0\\
            239 & 94 & 91 & 3\\
            241 & 96 & 92 & 4\\
            251 & 100 & 95 & 5\\
            269 & 106 & 106 & 0\\
            271 & 108 & 105 & 3\\
            281 & 112 & 110 & 2\\
            311 & 124 & 123 & 1\\
            331 & 132 & 129 & 3\\
            349 & 138 & 134 & 4\\
            359 & 142 & 139 & 3\\
            379 & 150 & 147 & 3\\
            389 & 154 & 154 & 0\\
            401 & 160 & 158 & 2\\
            409 & 162 & 156 & 6\\
            419 & 166 & 165 & 1\\
            421 & 168 & 162 & 6\\
            431 & 172 & 165 & 7\\
            439 & 174 & 169 & 5\\
            \hline
        \end{longtable}
    \end{center}
\end{rmk}

\bibliography{takei}
\bibliographystyle{alpha}

\end{document}